\documentclass[12pt]{amsart}

\usepackage{amssymb}
\usepackage{amscd}

\title{Algebraic curves admitting non-collinear Galois points} 
\author{Satoru Fukasawa}

\address{Department of Mathematical Sciences, Faculty of Science, Yamagata University, Kojirakawa-machi 1-4-12, Yamagata 990-8560, Japan} 
\email{s.fukasawa@sci.kj.yamagata-u.ac.jp}

\thanks{The author was partially supported by JSPS KAKENHI Grant Number 19K03438.}

\newtheorem{theorem}{Theorem}

\newtheorem{fact}{Fact}

\theoremstyle{definition}

\begin{document}
\begin{abstract} 
A criterion for the existence of a birational embedding into a projective plane with non-collinear Galois points for algebraic curves is presented. 
A new example of a plane curve with non-collinear Galois points as an application is described. 
Furthermore, a new characterization of the Fermat curve in terms of non-collinear Galois points is presented. 
\end{abstract}

\maketitle 

\section{Introduction} 
Let $X$ be a (reduced, irreducible) smooth projective curve over an algebraically closed field $k$ of characteristic $p \ge 0$ and let $k(X)$ be its function field. 
We consider a morphism $\varphi: X \rightarrow \mathbb{P}^2$, which is birational onto its image. 
A point $P \in \mathbb{P}^2$ is called a {\it Galois point}, if the field extension $k(\varphi(X))/\pi_P^*k(\mathbb{P}^1)$ of function fields induced by the projection $\pi_P$ from $P$ is a Galois extension. 
This notion was introduced by Hisao Yoshihara in 1996, to investigate the function fields of algebraic curves (\cite{miura-yoshihara, yoshihara}). 
The associated Galois group is denoted by $G_P$, when $P$ is a Galois point. 
Furthermore, a Galois point $P$ is said to be inner (resp. outer), if $P \in \varphi(X) \setminus {\rm Sing}(\varphi(X))$ (resp. if $P \in \mathbb{P}^2 \setminus \varphi(X)$). 
It is a natural and interesting problem to determine the number of Galois points for any curve $X$ and any birational embedding $\varphi: X \rightarrow \mathbb{P}^2$. 
However, this problem is difficult in general. 

Until recent years, it was not easy to construct a pair $(X, \varphi)$ such that $\varphi(X)$ admits two Galois points.  
In 2016, a criterion for the existence of a birational embedding with two Galois points was described by the present author (\cite{fukasawa1}), and by this criterion, a lot of new examples of plane curves with two Galois points were obtained (\cite{fukasawa1, open}). 
We recall this criterion. 

\begin{fact} \label{criterion} 
Let $G_1$ and $G_2$ be finite subgroups of ${\rm Aut}(X)$ and let $P_1$ and $P_2$ be different points of $X$.
Then, three conditions
\begin{itemize}
\item[(a)] $X/{G_i} \cong \Bbb P^1$ for $i=1, 2$,    
\item[(b)] $G_1 \cap G_2=\{1\}$, and
\item[(c)] $P_1+\sum_{\sigma \in G_1} \sigma (P_2)=P_2+\sum_{\tau \in G_2} \tau (P_1)$ in ${\rm Div}(X)$ 
\end{itemize}
are satisfied, if and only if there exists a birational embedding $\varphi: X \rightarrow \mathbb P^2$ of degree $|G_1|+1$ such that $\varphi(P_1)$ and $\varphi(P_2)$ are different inner Galois points for $\varphi(X)$ and the associated Galois group $G_{\varphi(P_i)}$ coincides with $G_i$ for $i=1, 2$. 
\end{fact}

It is a natural problem to find a condition for the existence of {\it non-collinear} Galois points (see also \cite{open}). 
This problem is solved as follows.

\begin{theorem} \label{main} 
Let $G_1$, $G_2$ and $G_3 \subset {\rm Aut}(X)$ be finite subgroups, and let $P_1$, $P_2$ and $P_3$ be different points of $X$. 
Then, four conditions
\begin{itemize}
\item[(a)] $X/{G_i} \cong \Bbb P^1$ for $i=1, 2, 3$,    
\item[(b)] $G_i \cap G_j=\{1\}$ for any $i, j$ with $i \ne j$,  
\item[(c)] $P_i+\sum_{\sigma \in G_i} \sigma (P_j)=P_j+\sum_{\tau \in G_j} \tau (P_i)$ for any $i, j$ with $i \ne j$, and
\item[(d)] $G_i P_j \ne G_i P_k$ for any $i, j, k$ with $\{i, j, k\}=\{1, 2,3\}$ 
\end{itemize}
are satisfied, if and only if there exists a birational embedding $\varphi: X \rightarrow \mathbb P^2$ of degree $|G_1|+1$ such that $\varphi(P_1)$, $\varphi(P_2)$ and $\varphi(P_3)$ are non-collinear inner Galois points for $\varphi(X)$ and $G_{\varphi(P_i)}=G_i$ for $i=1, 2, 3$. 
\end{theorem}

\begin{theorem} \label{main-outer}
Let $G_1$, $G_2$ and $G_3 \subset {\rm Aut}(X)$ be finite subgroups, and let $Q_1$, $Q_2$ and $Q_3$ be different points of $X$. 
Then, four conditions
\begin{itemize}
\item[(a)] $X/{G_i} \cong \Bbb P^1$ for $i=1, 2, 3$,    
\item[(b)] $G_i \cap G_j=\{1\}$ for any $i, j$ with $i \ne j$,  
\item[(c')] $\sum_{\sigma \in G_i}\sigma(Q_k)=\sum_{\tau \in G_j} \tau(Q_k)$ for any $i, j, k$ with $\{i, j, k\}=\{1, 2, 3\}$, and 
\item[(d')] $G_iQ_j \ne G_iQ_k$ for any $i, j, k$ with $\{i, j, k\}=\{1, 2, 3\}$
\end{itemize}
are satisfied, if and only if there exists a birational embedding $\varphi: X \rightarrow \mathbb P^2$ of degree $|G_1|$ and non-collinear outer Galois points $P_1, P_2$ and $P_3$ exist for $\varphi(X)$ such that $G_{P_i}=G_i$ and $\overline{P_iP_j} \ni \varphi(Q_k)$ for any $i, j, k$ with $\{i, j, k\}=\{1, 2, 3\}$, where $\overline{P_iP_j}$ is the line passing through $P_i$ and $P_j$.  
\end{theorem} 

As an application, a new example of a plane curve with non-collinear outer Galois points is constructed as follows. 

\begin{theorem} \label{hermitian} 
Let $p>0$, $q$ be a power of $p$, and let $X \subset \mathbb{P}^2$ be the Hermitian curve, which is (the projective closure of) the curve given by 
$$ x^q+x=y^{q+1}. $$
If a positive integer $s$ divides $q-1$, then there exists a plane model of $X$ of degree $s(q+1)$ admitting non-collinear outer Galois points $P_1, P_2$ and $P_3$. 
\end{theorem}

A next task is to classify plane curves with non-collinear Galois points. 
We consider the group $G :=\langle G_{P_1}, G_{P_2}, G_{P_3} \rangle \subset {\rm Aut}(X)$ for non-collinear outer Galois points $P_1, P_2$ and $P_3$.  
The case where the orbit $GQ$ of $Q$ is included in $\varphi^{-1}(\bigcup_{i \ne j}\overline{P_iP_j})$ for any $Q \in \varphi^{-1}(\bigcup_{i \ne j}\overline{P_iP_j})$ is determined as follows. 

\begin{theorem} \label{Fermat}
Let $\varphi: X \rightarrow \mathbb{P}^2$ be a birational embedding of degree $d \ge 3$ and let $C=\varphi(X)$. 
Then, the following conditions are equivalent. 
\begin{itemize}
\item[(a)] There exist non-collinear Galois points $P_1, P_2$ and $P_3 \in \mathbb{P}^2 \setminus C$ such that $GQ \subset \varphi^{-1}(\bigcup_{i \ne j}\overline{P_iP_j})$ for any $Q \in \varphi^{-1}(\bigcup_{i \ne j}\overline{P_iP_j})$, where $G=\langle G_{P_1}, G_{P_2}, G_{P_3} \rangle$.  
\item[(b)] $p=0$ or $d$ is prime to $p$, and $C$ is projectively equivalent to the Fermat curve $X^d+Y^d+Z^d=0$. 
\end{itemize}    
\end{theorem}

\section{Proof of Theorems \ref{main} and \ref{main-outer}} 

\begin{proof}[Proof of Theorem \ref{main}]
We consider the if-part. 
According to Fact \ref{criterion}, conditions (a), (b) and (c) are satisfied.  
Since points $\varphi(P_1), \varphi(P_2)$ and $\varphi(P_3)$ are not collinear, condition (d) is satisfied. 

We consider the only-if part. 
By condition (d), 
$$ \sum_{\sigma \in G_1}\sigma(P_2) \ne \sum_{\sigma \in G_1}\sigma(P_3). $$
Then, by condition (a), there exists a function $f \in k(X) \setminus k$ such that 
$$ k(X)^{G_1}=k(f), \ (f)=\sum_{\sigma \in G_1}\sigma(P_3)-\sum_{\sigma \in G_1}\sigma(P_2) $$
(see also \cite[III.7.1, III.7.2, III.8.2]{stichtenoth}). 
Similarly, there exists $g \in k(X) \setminus k$ such that 
$$ k(X)^{G_2}=k(g), \ (g)=\sum_{\tau \in G_2}\tau(P_3)-\sum_{\tau \in G_2}\tau(P_1). $$
Considering condition (c), we take a divisor 
$$ D:=P_1+\sum_{\sigma \in G_1}\sigma(P_2)=P_2+\sum_{\tau \in G_2}\tau(P_1). $$
Then, $f, g \in \mathcal{L}(D)$ and the sublinear system of $|D|$ corresponding to a linear space $\langle f, g, 1\rangle$ is base-point-free. 
Using condition (b), the induced morphism 
$$ \varphi: X \rightarrow \mathbb{P}^2; \ (f:g:1)$$
is birational onto its image, and points $\varphi(P_1)=(0:1:0)$ and $\varphi(P_2)=(1:0:0)$ are inner Galois points for $\varphi(X)$ such that $G_{\varphi(P_1)}=G_1$ and $G_{\varphi(P_2)}=G_2$ (see \cite[Proofs of Proposition 1 and of Theorem 1]{fukasawa1}). 
Furthermore, $\varphi(P_3)=(0:0:1)$. 
Using condition (c),
\begin{eqnarray*}
(g/f)&=& \sum_{\tau \in G_2}\tau(P_3)-\sum_{\tau \in G_2}\tau(P_1)-\sum_{\sigma \in G_1}\sigma(P_3)+\sum_{\sigma \in G_1}\sigma(P_2) \\
&=& (P_2+\sum_{\tau \in G_2}\tau(P_3))-(P_2+\sum_{\tau \in G_2}\tau(P_1)) \\
& &-(P_1+\sum_{\sigma \in G_1}\sigma(P_3))+(P_1+\sum_{\sigma \in G_1}\sigma(P_2)) \\
&=&(P_3+\sum_{\gamma \in G_3}\gamma(P_2))-(P_3+\sum_{\gamma \in G_3}\gamma(P_1)) \\
&=&\sum_{\gamma \in G_3}\gamma(P_2)-\sum_{\gamma \in G_3}\gamma(P_1).  
\end{eqnarray*}
Then, the subfield $k(g/f)$ induced by the projection from $P_3$ coincides with $k(X)^{G_3}$. 
Therefore, this point $\varphi(P_3)$ is inner Galois with $G_{\varphi(P_3)}=G_3$. 
\end{proof}

The proof of Theorem \ref{main-outer} is very similar.

\section{A new example} 
Let $X \subset \mathbb{P}^2$ be the Hermitian curve of degree $q+1$. 
The set of all $\mathbb{F}_{q^2}$-rational points of $X$ is denoted by $X(\mathbb{F}_{q^2})$. 
See \cite{hkt} for properties of the Hermitian curve. 

\begin{proof}[Proof of Theorem \ref{hermitian}]
Let $Q_1=(1:0:0)$ and $Q_2=(0:0:1)$, and let $Q_3 \in X(\mathbb{F}_{q^2})$ with $Q_3 \not\in \overline{Q_1Q_2}=\{Y=0\}$.  
Then, the matrix 
$$ A_{a}:=\left(\begin{array}{ccc}
a^{q+1} & 0 & 0\\
0 & a & 0 \\
0 & 0 & 1 
\end{array}\right)$$
acts on $X$ and fixes $Q_1$ and $Q_2$, where $a \in \mathbb{F}_{q^2} \setminus \{0\}$. 
Let $s m=q-1$ and let $G_3 \subset {\rm Aut}(X)$ be the cyclic group of order $s(q+1)$ consisting of all $A_{a^m}$. 
Note that each element of $G_3 \setminus \{1\}$ does not fix $Q_3$. 
Considering the Sylow $p$-group of ${\rm Aut}(X)$ fixing $Q_1$, it follows that there exists an automorphism $\Phi \in {\rm Aut}(X)$ such that $\Phi(Q_1)=Q_1$, $\Phi(Q_2)=Q_3$ and $\Phi(Q_3)=Q_2$. 
Then, the group $\Phi G_3 \Phi^{-1}$ fixes points $Q_1$ and $Q_3$, and each element of this group different from identity does not fix $Q_2$. 
Therefore, for each pair $(Q_i, Q_j)$, there exists a cyclic group $G_k$ of order $s(q+1)$ such that $G_k$ fixes points $Q_i$ and $Q_j$ and each element of $G_k \setminus \{1\}$ does not fix $Q_k$. 
We would like to show that conditions (a), (b), (c') and (d') in Theorem \ref{main-outer} are satisfied for groups $G_1, G_2$ and $G_3$. 

Note that 
$$ (A_{a^m})^s=A_{a^{q-1}}=
\left(\begin{array}{ccc} 
1 & 0 & 0 \\
0 & a^{q-1} & 0 \\
0 & 0 & 1 
\end{array}\right). $$
Let $G_3' \subset G_3$ be a subgroup consisting of all $A_{a^{q-1}}$.
Since $k(X)^{G_3} \subset k(X)^{G_3'}=k(x)$, by L\"{u}roth's theorem, $X/G_3$ is rational. 
Condition (a) is satisfied. 
Since $G_1$ fixes $Q_2$ and the set $G_2 \setminus \{1\}$ does not contain an element fixing $Q_2$, $G_1 \cap G_2=\{1\}$. 
Condition (b) is satisfied. 
For any $i, j, k$ with $\{i, j, k\}=\{1, 2, 3\}$, 
$$ \sum_{\sigma \in G_i} \sigma(Q_k)=s(q+1)Q_k=\sum_{\tau \in G_j}\tau(Q_k). $$
Condition (c') is satisfied. 
Since $G_iQ_j=\{Q_j\} \ne \{Q_k\}=G_iQ_k$, condition (d') is satisfied. 
\end{proof}

\section{A characterization of the Fermat curve}

\begin{proof}[Proof of Theorem \ref{Fermat}] 
(a) $\Rightarrow$ (b). 
Let $Q \in \varphi^{-1}(\overline{P_1P_2})$. 
By the definition of outer Galois points, $G_{P_1} Q \subset \varphi^{-1}(\overline{P_1P_2})$, $G_{P_2} Q \subset \varphi^{-1}(\overline{P_1P_2})$ and $G_{P_3}Q \subset \varphi^{-1}(\overline{P_3\varphi(Q)})$. 
If $\gamma (Q) \in \varphi^{-1}(\overline{P_2P_3})$ for some $\gamma \in G_{P_3}$, then $\varphi(\gamma(Q)) \in \overline{P_3\varphi(Q)} \cap \overline{P_2P_3}=\{P_3\}$.
This is a contradiction. 
Therefore, condition (a) implies that $G_{P_3}Q \subset \varphi^{-1}(\overline{P_1P_2})$. 
It follows that $\gamma \in G_{P_3}$ induces a bijection of ${\rm supp}(\varphi^*\overline{P_1P_2})$. 
Since $G_{P_1}$ acts on ${\rm supp}(\varphi^*\overline{P_1P_2})$ transitively, 
$$\varphi^*\overline{P_1P_2}=\sum_{Q \in {\rm supp}(\varphi^*\overline{P_1P_2})}mQ$$ 
for some integer $m \ge 1$. 
Therefore, 
$$\gamma^*\varphi^*(\overline{P_1P_2})=\varphi^*(\overline{P_1P_2}).$$ 
Let $D:=\varphi^*\overline{P_3P_1}$. 
We take a function $f \in k(X)$ with $k(f)=k(X)^{G_3}$ such that 
$$(f)=\varphi^*\overline{P_3P_2}-\varphi^*\overline{P_3P_1}.  $$
Similarly, we can take a function $g \in k(X)^{G_1}$ such that 
$$(g)=\varphi^*\overline{P_1P_2}-\varphi^*\overline{P_1P_3}.$$ 
Since $\overline{P_1P_2}$ does not pass through $P_3$, $g \not \in \langle f, 1 \rangle \subset \mathcal{L}(D)$. 
It follows from the condition $\gamma^*\varphi^*(\overline{P_1P_2})=\varphi^*(\overline{P_1P_2})$ that $\gamma^*g=a(\gamma) g$ for some $a(\gamma) \in k$. 
Therefore, a linear subspace $\langle f, 1, g \rangle \subset \mathcal{L}(D)$ is invariant under the action of $\gamma \in G_{P_3}$.  
Since $\varphi$ is represented by $(f:1:g)$, there exists an injective homomorphism 
$$ G_{P_3} \hookrightarrow {\rm PGL}(3, k); \ \gamma \mapsto 
\left(\begin{array}{ccc}
1 & 0 & 0\\
0 & 1 & 0 \\
0 & 0 & a(\gamma)
\end{array}\right). $$
It follows that $d$ is prime to $p$, and the map $G_{P_3} \rightarrow k \setminus \{0\}$; $\gamma \mapsto a(\gamma)$ is an injective homomorphism. 
This implies that $G_{P_3}$ is a cyclic group, and $C$ is invariant under the linear transformation $(X:Y:Z) \mapsto (X:Y:\zeta Z)$, where $\zeta$ is a primitive $d$-th root of unity.  
Similarly, $G_{P_1}$ is generated by the automorphism given by the linear transformation $(X:Y:Z) \mapsto (\zeta X:Y:Z)$. 
It follows that $C$ is defined by $X^d+Y^d+Z^d=0$. 

(b) $\Rightarrow$ (a). 
This is derived from the fact that groups $G_{P_1}$, $G_{P_2}$ and $G_{P_3}$ fix all points on the lines $\{X=0\}$, $\{Y=0\}$ and $\{Z=0\}$ respectively for the Fermat curve, where $P_1=(1:0:0)$, $P_2=(0:1:0)$ and $P_3=(0:0:1)$. 
\end{proof}

\begin{center} {\bf Acknowledgements} \end{center} 
The author is grateful to Doctor Kazuki Higashine for helpful discussions.

\end{document}